\documentclass[12pt,reqno]{amsart}
\usepackage{amssymb}
\usepackage[all]{xy}
\usepackage{amsmath}

 \rm

\renewcommand{\(}{\left(}
\renewcommand{\)}{\right)}
\renewcommand{\[}{\left[}
\renewcommand{\]}{\right]}
\newcommand{\<}{\langle}
\renewcommand{\>}{\rangle}

\renewcommand{\bar}{\overline}
\newcommand{\abs}[1]{\left\lvert#1\right\rvert}
\newcommand{\norm}[1]{\left\lVert#1\right\rVert}

\newcommand{\C}{{\mathbb{C}}}
\newcommand{\R}{{\mathbb{R}}}

\newcommand{\N}{{\mathbb{N}}}

\renewcommand{\phi}{\varphi}

\theoremstyle{plain}
\newtheorem{thm}{Theorem}[section]
\newtheorem{lem}[thm]{Lemma}
\newtheorem{prop}[thm]{Proposition}
\newtheorem{cor}[thm]{Corollary}

\theoremstyle{definition}

\theoremstyle{remark}
\newtheorem{rem}[thm]{Remark}

\title[On the Carlitz-Mehler formula]{On the Carlitz-Mehler formula for Hermite polynomials}

\author{Manish Chaurasia}
\address{School of Mathematical Sciences, NISER Bhubaneswar, 
 Jatni, 752050.}
\email{mchaurasia@niser.ac.in}

\begin{document}

\begin{abstract}
Carlitz proved a few generalizations of Mehler's formula. Later, Srivastava et al. gave a new proof to some extensions of Carlitz's formula. Here, a direct proof to the further generalization is given.
\end{abstract}

\maketitle

%\tableofcontents\newpage

\section{Introduction}\label{intro}

The Hermite polynomials can be defined by the generating series 
\begin{equation*}
\sum_{n=0}^{\infty} H_n(x) \frac{t^n}{n!} = \exp\(2xt-t^2\). 
\end{equation*}
The following bilinear Hermite expansion 
\begin{equation*}
\sum_{n=0}^{\infty} H_n(x_1) H_n(x_2) \frac{t^n}{n!}
 = \(1-4t^2\)^{-1/2}
\exp\(\frac{4x_1x_2t-4(x_1^2+x_2^2)t^2}{1-4t^2}\)
\end{equation*}
is known as Mehler's formula (see \cite{EDR}).  In \cite{Carlitz2,Carlitz1}, Carlitz proved a 
few extensions of the Mehler's formula. 
Out of those elegant formulas, two essential formulas are the following.
\begin{equation*}\label{E:f1}
\begin{split}
\sum_{n=0}^{\infty} H_{n+r}(x_1) H_{n+s}(x_2) 
\frac{t^n}{n!} =\(1-4t^2\)^{-(r+s+1)/2}
\exp\(\frac{4x_1x_2t-4(x_1^2+x_2^2)t^2}{1-4t^2}\) \\
\cdot\sum_{k=0}^{min(r,s)} 2^{2k} k! {r \choose k} 
{s \choose k} t^k H_{r-k}\(\frac{x_1-2x_2t}
{\sqrt{1-t^2}}\) H_{s-k}\(\frac{x_2-2x_1t}{\sqrt{1-t^2}}\).
\end{split}
\end{equation*}
\begin{equation*}\label{E:f2}
\begin{split}
\sum_{m,n,p=0}^{\infty} H_{n+p}(x_1) H_{p+m}(x_2)
H_{m+n}(x_3)\frac{u_1^m}{m!}\frac{u_2^n}{n!}\frac{u_3^p}{p!}
\qquad \qquad \qquad \qquad \qquad 
\qquad \qquad \qquad\\ 
= {\Delta}^{-1/2} \exp\(\sum x_i^2 - \frac{\sum x_i^2 - 4\sum x_i^2 u_i^2 - 
4\sum_{i\neq j\neq k} x_i x_j u_k + 8\sum_{i\neq j} u_i u_j x_i x_j}{\Delta}\),
\end{split}
\end{equation*}
where $\Delta = 1-4u_1^2-4u_2^2-4u_3^2+16u_1u_2u_3$.

In \cite{JPS,SS}, Srivastava and Singal
proved some extensions of the above formulas.
 %(\ref{E:f1}) and (\ref{E:f2})
For $x=(x_1,x_2,x_3)\in\R^3$, let 
$X_i(x) = \frac{(1-4u_i^2)x_i-2
\sum_{i\neq j\neq k\neq i}(u_k-2u_iu_j)x_j}{\sqrt{\Delta(1-4u_i^2)}}.$
Let $D_if=\frac{df}{dx_i}$. For $2\le d\in\N$, let $r_1,r_2,\dots,r_{d}\in\N$, 
and let 
$$
\Lambda_d (r_1,r_2,\dots,r_d)
= \Big\{\(k_1,k_2,\dots,k_{d-1}\)\in(\N\cup\{0\})^{d-1} : 
\sum_{i=1}^{d-1}k_i\le \min\{r_1,\dots,r_d\} \Big\}.
$$
We shall write $\Lambda_d$ at the place of 
$\Lambda_d (r_1,r_2,\dots,r_d)$.
With these, one of the main extension of \cite{SS} can be stated in the following way.
\begin{equation}\label{E:srivastava}
\begin{split}
\sum_{m,n,p=0}^{\infty} H_{n+p+r_1}(x_1) H_{p+m+r_2}(x_2)
H_{m+n}(x_3)\frac{u_1^m}{m!}\frac{u_2^n}{n!}\frac{u_3^p}{p!}\qquad\qquad\qquad\qquad\qquad \qquad 
\qquad \qquad \qquad
\qquad \qquad \qquad\quad\\ 
= {\Delta}^{-(r_1+r_2+1)/2} (1-4u_1^2)^{r/2}(1-4u_2^2)^{s/2}
\qquad\qquad\qquad\qquad\qquad \qquad 
\qquad\qquad\qquad\qquad\qquad
\qquad\qquad\quad\qquad\\
\cdot \exp\(\sum x_i^2 - \frac{\sum x_i^2 - 4\sum x_i^2 u_i^2 - 
4\sum_{i\neq j\neq k} x_i x_j u_k + 8\sum_{i\neq j} u_i u_j x_i x_j}{\Delta}\)
\qquad\qquad\qquad\quad
\qquad\qquad\qquad\,\,\,\,\\
\cdot \sum_{k\in\Lambda_2} 2^{2k_1} k_1! {r \choose k_1} 
{s \choose k_1}\(\frac{u_3-2u_1u_1}{(1-4u_1^2)(1-4u_2^2)^{1/2}}\)^{k_1}
H_{r-k_1}\(X_1(x)\)H_{s-k_1}\(X_2(x)\).
\qquad\qquad\qquad\qquad\qquad\,\,
\end{split}
\end{equation}
These formulas came to the attention when we became interested in
weighted $l^k$-sum of Hermite functions, which is recently explored
by Radchenko-Ramos \cite{RR}. 
In this paper, we shall give a direct proof to the following generalization
of the formulas mentioned above, which is the first main result of this paper.
\begin{equation}\label{E:GCMF}
\begin{split}
\sum_{m,n,p=0}^{\infty} H_{n+p+r_1}(x_1) H_{p+m+r_2}(x_2)
H_{m+n+r_3}(x_3)\frac{u_1^m}{m!}\frac{u_2^n}{n!}\frac{u_3^p}{p!}
\qquad\qquad\qquad\qquad\qquad \qquad 
\qquad \qquad \qquad
\qquad \qquad \qquad\quad\\ 
= \prod_{i=1}^3 \Delta^{-\(\frac{r_i+1}{2}\)}(1-4u_i^2)^{\frac{r_i}{2}}
\qquad\qquad\qquad\qquad\qquad \qquad 
\qquad \qquad \qquad
\qquad \qquad \qquad\quad
\qquad\qquad\qquad\qquad\quad\,\qquad\\ 
\cdot \exp\(\sum x_i^2 - \frac{\sum x_i^2 - 4\sum x_i^2 u_i^2 - 
4\sum_{i\neq j\neq k} x_i x_j u_k + 8\sum_{i\neq j} u_i u_j x_i x_j}{\Delta}\)
\qquad\qquad\quad\qquad \qquad \qquad
\qquad\qquad\,\,\\
\cdot \sum_{k\in\Lambda_3} \prod_{i=1}^{2} 
\(- \sqrt{\frac{\Delta}{2(1-4u_i^2)}}\)^{k_i} {r_{i+1} \choose k_i}
H_{r_3-k_2}(X_3(x))D_3^{k_2}\[
H_{r_2-k_1}(X_2(x))D_{2}^{k_1} H_{r_1}\(X_1(x)\)\].
\qquad\qquad\qquad\qquad\,\,\,
\end{split}
\end{equation}
%We would like to mention that such formulas increase the understanding of Hermite
%series, which is not expected from an arbitrary series.

For any $j\in\{1,2,\dots,d\}$, let $A_j^* 
= \(-\frac{d}{dx_j}+2x_j\)$ be the creation operators.  
Our second main result is about
the action of creation operators on a dilated gaussian.
%which could be of independent interest. 

\begin{thm}\label{T:creation}
Let $A = \(a_{ij}\)_{1\le i,j \le d}$
be a real symmetric matrix.  
For $i\in\{1,2,\dots,d\}$, let $S_i(x) = \frac{(1+a_{ii})x_i+
\sum_{i\neq j}a_{ij}x_j}{\sqrt{1+a_{ii}}}$, and for $k\in\Lambda_d$, 
let
\begin{equation*}
\begin{split}
P_k(x)
\qquad\qquad\qquad\qquad\qquad \qquad 
\qquad \qquad \qquad
\qquad \qquad \qquad\qquad
\qquad\qquad\qquad\qquad\quad\qquad\quad\,\,\,\\ 
 = H_{r_{d}-k_{d-1}}(S_d(x))D_d^{k_{d-1}}\[H_{r_{d-1}-k_{d-2}}\dots
D_4^{k_3}\[H_{r_3-k_2}(S_3(x))D_3^{k_2}\[
H_{r_2-k_1}(S_2(x))D_{2}^{k_1} H_{r_1}(S_1(x))\]\]\].
\end{split}
\end{equation*}
Then, for $d \ge 2$
\begin{equation}\label{E:dilatedhf}
(A_d^*)^{r_d}\dots(A_2^*)^{r_2}(A_1^*)^{r_1} e^{-x^tAx}
= \prod_{i=1}^d(1+a_{ii})^{\frac{r_i}{2}}
   \sum_{k\in\Lambda_d} \prod_{i=1}^{d-1} 
   \frac{(-1)^{k_i}}{(1+a_{ii})^{\frac{k_i}{2}}} {r_{i+1} \choose k_i}
   P_k(x) e^{-x^tAx},
%{k_1,k_2\dots,k_{n-1}=0}^{k_1+k_2+\dots+k_{n-1}
%\le\min\{r_1,r_2,\dots,r_n\}}
\end{equation}
and for $d=1$
\begin{equation*}
(A_1^*)^{r_1} e^{-x^tAx}
=  (1+a_{11})^{\frac{r_1}{2}}H_{r_1}\(S_1(x)\)e^{-x^tAx}.
\end{equation*}
\end{thm}
 To prove it, we shall use some 
combinatorial properties of Hermite polynomials.
As an application of this result, along with the Bargmann transform,
we shall prove formula (\ref{E:GCMF}).

\section{Preliminaries}
%First, we collect some useful tools. 
For $a,b\in\R$, let $T_af(x) = f(x+a)$, $M_bf(x) = f(bx)$,
and $Df = \frac{df}{dx}$.
For any nonnegative integer $n$, the $n$-th Hermite polynomial
$H_n(x)$ can also be defined by
\begin{equation*}\label{E:HPdef}
H_n(x) = (-1)^n e^{x^2} D^n e^{-x^2}.
\end{equation*}
We define the $n$-th dilated Hermite function in the following way.
\begin{equation*}\label{E:HF}
h_n(x) = \(\frac{2}{\pi}\)^{1/4} 
\frac{e^{-x^2}H_n(\sqrt{2}x)}{\sqrt{2^n n!}}.
\end{equation*}
The following two identities for the Hermite polynomials are needed 
(see \cite{JLB,MO}).
%A crucial formula in the theory of Hermite polynomials is the 
%Burchnall’s operational formula \cite{Burchnall}:
\begin{equation}\label{E:BOF}
\(-D+2x\)^m(f) = m!\sum_{j=0}^{m} \frac{(-1)^j}{j!}\
\frac{H_{m-j}(x)}{(m-j)!} D^j (f).
\end{equation}
%The following identity is known as 
%the addition formula for Hermite polynomials, see \cite{AFHP}. 
\begin{equation}\label{E:AF}
{\displaystyle {\frac {\left(\sum _{k=1}^{r}a_{k}^{2}\right)^
{\frac {n}{2}}}{n!}}H_{n}
\left({\frac {\sum _{k=1}^{r}a_{k}x_{k}}
{\sqrt {\sum _{k=1}^{r}a_{k}^{2}}}}\right)
=\sum _{m_{1}+m_{2}+\ldots +m_{r}=n,m_{i}\geq 0}
\prod _{k=1}^{r}\left\{{\frac {a_{k}^{m_{k}}}{m_{k}!}}
H_{m_{k}}\left(x_{k}\right)\right\}}
\end{equation}
Formulas (\ref{E:BOF}) and (\ref{E:AF}) are known as the 
operational and addition formula respectively.  

For $z=(z_1,\dots,z_d)\in\C^d$, let 
$\norm{z}^2 = z_1^2+\dots+z_d^2$.
%For $n=(n_1,\dots,n_d)\in\N^d\cup\{0\}$, and 
%$x=(x_1,\dots,x_d)\in\R^d$, let $\Phi_n(x) 
%= h_{n_1}(x_1)\dots h_{n_d}(x_d)$.
%The collection $\{\Phi_n\}$ forms an orthonormal basis for $L^2(\R^d)$.
Let $F^2(\C^d)$ be the inner product space of all entire functions 
on $\C^d$ with
\begin{equation*}
\< F, G \> = \int_{\C^n} F(z) \bar{G(z)} e^{-\norm{z}^2}
\frac{dz}{{\pi}^d}.
\end{equation*}
For $f\in \mathcal{S}(\R^d)$, we define the Bargmann transform 
in the following way.
\begin{equation*}\label{E:BT}
Bf(z) = \(\frac{2}{\pi}\)^{d/4} e^{-\frac{1}{2}\norm{z}^2}
\int_{\R^n} e^{2x \cdot z - \abs{x}^2}
 f(x)\,dx.
\end{equation*}
Then $B$ extends to an isometric isomorphism from $L^2(\R^d)$ 
to $F^2(\C^d)$. Moreover, for $F\in F^2(\C^d)$, the inverse 
Bargmann transform is given by
\begin{equation}\label{E:inverse}
B^{-1}F(x) = \(\frac{2}{\pi}\)^{d/4} e^{-\abs{x}^2}
\int_{\C^d} e^{2x \cdot \bar{z} - \frac{1}{2} \norm{\bar{z}}^2}
F(z)e^{-\norm{z}^2} \frac{dz}{{\pi}^d}.
\end{equation}
%and 
%\begin{equation}\label{E:BTisometry}
%Bh_{n_j}(z_j) = \frac{z_j^n}{\sqrt{n_j!}},
%\end{equation}
%for $j\in\{1,2,\dots,d\}$.
%, let $g(x_j) = f(x_1,\dots,x_j,\dots,x_d)$, 
%and 
For $j\in\{1,2,\dots,d\}$, define
\begin{equation*}\label{E:BTcomponent}
B_jf(w) = \(\frac{2}{\pi}\)^{1/4} e^{-\frac{1}{2} w^2}
\int_{\R} e^{2x_jw - x_j^2}
 f(x_1,\dots,x_j\dots,x_d)\,dx_j.
\end{equation*}
The map $B_j$ extends to an isometric isomorphism from 
$L^2(\R)$ to $F^2(\C)$, and 
\begin{equation*}\label{E:CI}
B_jh_{n}(w) = \frac{w^n}{\sqrt{n!}}.
\end{equation*}
The image of creation operator under the 
transformation $B_j$ is well-known (see \cite{KZ}, Lemma 6.13),
which will be essential for us. 
%It is well-known (see \cite{KZ}) that
\begin{equation}\label{E:image}
B_jA_j^*f(w) = 2^jwB_jf(w)
%\quad \text{and} \quad  
%B_jA_jf(w) = \frac{d}{dw}B_jf(w).
\end{equation}
Observe that, if $f(x_1,\dots,x_d) = \prod_{j=1}^d f_j(x_j)$, for some
$f_j$, then 
\begin{equation}\label{E:BTdecomposition}
Bf(z) = \prod_{j=1}^d B_jf_j(z_j).
\end{equation} 
The Bargmann transform and its properties can be found in \cite{KZ}.
\section{Proofs}
\subsection{Proof of Theorem \ref{T:creation}}
%We take $d=3$ for convienience. Repetation of same argument
%gives the general case. 
We proceed with the following proposition.
%Let $g_1(x_2,x_3) = e^{-a_{22}x_2^2-a_{33}x_3^2-2a_{23}x_2x_3}$,
%and $g_2(x_1,x_3) = e^{-a_{11}x_1^2-a_{33}x_3^2-2a_{13}x_1x_3}$.

\begin{prop}
Let $a,b\in \R$, and $a>0$. Then for $m\in\N$
\begin{equation*}
D^m e^{-at^2-bt}
= (-1)^m a^{\frac{m}{2}}\,
   H_m\(\sqrt{a}t+\frac{b}{2\sqrt{a}}\)e^{-at^2-bt}.
\end{equation*}
\end{prop}
\begin{proof}
%Write 
%$e^{-x^tAx} 
%= g_1(x_2,x_3)e^{-a_{11}x_1^2-2(a_{12}x_2+a_{13}x_3)x_1}$.
We compute
\begin{equation*}
\begin{aligned}
D^m e^{-at^2-bt}
%  =&\;   g_1\,D_1^m e^{-a_{11}x_1^2-2(a_{12}x_2+a_{13}x_3)x_1}\\
  =&\;    D^m\[e^{\frac{b^2}{4a}}
            T_{\frac{b}{2a}} M_{\sqrt{a}}
            e^{-t^2}\]\\
  =&\;    a^{\frac{m}{2}}\,
            e^{\frac{b^2}{4a}}
            T_{\frac{b}{2a}} 
            M_{\sqrt{a}}D^m e^{-t^2}\\
  =&\;  (-1)^m a^{\frac{m}{2}}\,e^{\frac{b^2}{4a}}
            T_{\frac{b}{2a}}
            M_{\sqrt{a}} H_m(t)e^{-t^2}\\
  =&\;   (-1)^m a^{\frac{m}{2}}\,
             e^{\frac{b^2}{4a}}H_m\(\sqrt{a}
            \(t+\frac{b}{2a}\)\)
            e^{-a \(t+\frac{b}{2a}\)^2}\\
  =&\;  (-1)^m a^{\frac{m}{2}}\,
           H_m\(\sqrt{a}t+\frac{b}{2\sqrt{a}}\)e^{-at^2-bt}. 
\end{aligned}
\end{equation*}
\end{proof}
\begin{cor}\label{C:DG}
Let $j\in\{1,2,\dots,d\}$.
If $A$ is the matrix as in Theorem \ref{T:creation}, then
\begin{equation*}
D_{i}^m e^{-x^tAx}
= (-1)^m a_{ii}^{\frac{m}{2}}
   H_m\(\sqrt{a_{ii}} \, x_i + \frac{\sum_{j=1, j\neq i}^d a_{ij} x_j}
   {\sqrt{a_{ii}}}\)e^{-x^tAx}.
\end{equation*}
\end{cor}
%\begin{rem}
%Similar results for the other differential operators $D_j$, for 
%$j\in\{2,3,\dots,d\}$, can be obtained by the same method. 
%\end{rem}
%Let $S_1(x) = \sqrt{1+a_{11}}x_1-\frac{a_{12}}
%{\sqrt{1+a_{11}}}x_2-\frac{a_{13}}{\sqrt{1+a_{11}}}x_3$.

For $d=1$, by Corollary \ref{C:DG}, formula (\ref{E:BOF})
and (\ref{E:AF}), we obtain
\begin{equation*}\label{E:onecreation}
\begin{aligned}
(A_1^*)^{r_1} e^{-x^tAx} 
=&\;  r_1! \sum_{m=0}^{r_1} \frac{(-1)^m}{m!} 
        \frac{H_{r_1-m}(x_1)}{(r_1-m)!} D_1^m 
        e^{-x^tAx} \\
=&\;  r_1! \sum_{m=0}^{r_1} \frac{(\sqrt{a_{11}})^m}{m!} 
         \frac{H_{r_1-m}(x_1)}{(r_1-m)!} 
         H_m\(\sqrt{a_{11}} \, x_1 + 
        \frac{\sum_{j=2}^d a_{1j} x_j}{\sqrt{a_{11}}}\)e^{-x^tAx}\\
=&\; (1+a_{11})^{\frac{r_1}{2}} 
        H_{r_1}\(\frac{x_1+\sqrt{a_{11}}
        \(\sqrt{a_{11}} \, x_1 + \frac{\sum_{j=2}^d a_{1j} x_j}          
        {\sqrt{a_{11}}}\)}{\sqrt{1+a_{11}}}\)e^{-x^tAx}\\
=&\;  (1+a_{11})^{\frac{r_1}{2}}H_{r_1}\(S_1(x)\)e^{-x^tAx}.
\end{aligned}
\end{equation*}
For $d=2$, by Corollary \ref{C:DG}, formula (\ref{E:BOF}) and 
(\ref{E:AF}), we compute
%= (A_2^*)^{r_2}\[g_1 r_1! \sum_{m=0}^{r_1} \frac{(-1)^m}{m!} 
%        \frac{H_{r_1-m}(x_1)}{(r_1-m)!} D_1^m 
%        e^{-a_{11}x_1^2+2(a_{12}x_2+a_{13}x_3)x_1}\]
%\qquad\qquad\qquad\qquad\qquad\qquad
%\qquad\quad\,\,\,\,\,\\
%=  (A_2^*)^{r_2}\[r_1! \sum_{m=0}^{r_1} \frac{(\sqrt{a_{11}})^m}{m!} 
%         \frac{H_{r_1-m}(x_1)}{(r_1-m)!} 
%         H_m\(\sqrt{a_{11}}x_1-\frac{a_{12}}{\sqrt{a_{11}}}x_2-
%         \frac{a_{13}}{\sqrt{a_{11}}}x_3\)e^{-x^tCx}\]
%\qquad\qquad\qquad\qquad\,\,\\
%= (A_2^*)^{r_2}\[(1+a_{11})^{\frac{r_1}{2}} 
%    H_{r_1}\(\frac{x_1+\sqrt{a_{11}}(\sqrt{a_{11}}x_1-\frac{a_{12}}
%    {\sqrt{a_{11}}}x_2-
%    \frac{a_{13}}{\sqrt{a_{11}}}x_3)}{\sqrt{1+a_{11}}}\)e^{-x^tCx}\]
%\qquad\qquad\qquad\qquad\qquad\,\,\,\,\\
%=  (A_2^*)^{r_2}\[(1+a_{11})^{\frac{r_1}{2}} 
%    H_{r_1}\(S_1(x)\)e^{-x^tCx}\]
%\qquad\qquad\qquad\qquad\qquad\qquad
%\qquad\qquad\quad\qquad\qquad\quad
%\qquad\qquad\\
\begin{equation*}
\begin{split}
(A_2^*)^{r_2}(A_1^*)^{r_1} e^{-x^tAx}
%\qquad\qquad\qquad\qquad\qquad\qquad
%\qquad\qquad\qquad\qquad\qquad\qquad
%\qquad\qquad\qquad\qquad\qquad\qquad
%\qquad\quad\\
=     (1+a_{11})^{\frac{r_1}{2}}r_2!
        \sum_{m=0}^{r_2} \frac{(-1)^m}{m!} 
        \frac{H_{r_2-m}(x_2)}{(s-m)!}
        D_2^m \(H_{r_1}\(S_1(x)\)
        e^{-x^tAx}\)
\qquad\qquad\qquad\qquad\qquad\quad\,\,\,\,\,\,\,
\end{split}
\end{equation*}
\begin{equation*}
\begin{split}
%(A_2^*)^{r_2}(A_1^*)^{r_1} e^{-x^tCx}
%\qquad\qquad\qquad\qquad\qquad\qquad
%\qquad\qquad\qquad\qquad\qquad\qquad
%\qquad\qquad\qquad\qquad\qquad\qquad
%\qquad\quad\\
=   (1+a_{11})^{\frac{r_1}{2}}r_2!\sum_{m=0}^{r_2} 
     \frac{(-1)^m}{m!}\frac{H_{r_2-m}(x_2)}{(s-m)!} 
%\qquad\qquad\qquad\qquad\qquad\qquad
%\qquad\qquad\qquad\qquad\quad\,\,\,\,\\
     \sum_{k_1=0}^m {m \choose k_1} 
      D_{2}^{k_1} H_{r_1}\(S_1(x)\)
     D_{2}^{m-k_1}e^{-x^tAx}
\qquad\qquad\qquad\qquad\quad
\qquad\qquad\qquad\,
\end{split}
\end{equation*}
\begin{equation*}
\begin{split}
=   (1+a_{11})^{\frac{r_1}{2}}r_2!
     \sum_{k_1=0}^{\min\{r_1,r_2\}} 
     \frac{(-1)^{k_1}}{k_1!} D_{2}^{k_1} H_{r_1}\(S_1(x)\)
\qquad\qquad\qquad\qquad\qquad\qquad
\qquad\qquad\qquad\qquad\qquad\quad\,\,\\
     \cdot\sum_{m=0}^{r_2} \frac{1}{(r_2-m)!} 
     \frac{(\sqrt{a_{22}})^{m-k_1}}{(m-k_1)!}H_{r_2-m}(x_2)
     H_{m-k_1}\(\sqrt{a_{22}} \, x_i + 
     \frac{\sum_{j=1, j\neq 2}^d a_{2j} x_j}{\sqrt{a_{22}}}\)e^{-x^tAx}
\qquad\qquad\qquad\quad\quad\,\,
\end{split}
\end{equation*}
\begin{equation*}
\begin{split}
=  (1+a_{11})^{\frac{r_1}{2}}r_2!\sum_{k_1=0}^{\min\{r_1,r_2\}} 
      \frac{(-1)^{k_1}}{k_1!} D_{2}^{k_1} H_{r_1}\(S_1(x)\)
\qquad\qquad\qquad\qquad\qquad\qquad
\qquad\qquad\qquad\qquad\qquad\qquad\,\,\,\\
    \cdot\frac{(1+a_{22})^{\frac{r_2-k_1}{2}}}{(r_2-k_1)!}
     H_{r_2-k_1}\(\frac{x_2+\sqrt{a_{22}}\(\sqrt{a_{22}} \, x_i + 
     \frac{\sum_{j=1, j\neq 2}^d a_{2j} x_j}{\sqrt{a_{22}}}\)}   
     {\sqrt{1+a_{22}}}\)e^{-x^tAx}
\qquad\qquad\qquad\qquad\qquad\qquad\quad\,\,
\,\,\,\,
\end{split}
\end{equation*}
\begin{equation*}
\begin{split}
=  (1+a_{11})^{\frac{r_1}{2}}(1+a_{22})^{\frac{r_2}{2}}
     \sum_{k_1=0}^{\min\{r_1,r_2\}} 
     \frac{(-1)^{k_1}}{(1+a_{22})^{\frac{k_1}{2}}} {r_2 \choose k_1}
     H_{r_2-k_1}(S_2(x))D_{2}^{k_1} H_{r_1}\(S_1(x)\)e^{-x^tAx}.
\qquad\qquad\qquad\,\,\,\,
\end{split}
\end{equation*}
In general, assuming the formula 
(\ref{E:dilatedhf}) for $d=p-1$, for $2<p\in\N$, and 
showing that the formula holds for $d=p$ is exactly same
as the case done. 
%We avoid writting that much complexity in the paper because it is evident. 
%One can check that.
\qed
%This completes the proof of of Theorem \ref{T:creation} for $d=3$.
%General case follows by just repeating the above argument.
%\begin{equation*}
%(A_n^*)^{r_d}\dots(A_2^*)^{r_2}(A_1^*)^{r_1} e^{-x^tCx}
%= \prod_{i=1}^d(1+a_{ii})^{\frac{r_i}{2}}
%   \sum_{k\in\Lambda_d} \prod_{i=1}^{d-1} 
%   \frac{(-1)^{k_i}}{(1+a_{ii})^{\frac{k_i}{2}}} {r_{i+1} \choose k_i}
%   P_k(x) e^{-x^tCx}.
%%{k_1,k_2\dots,k_{n-1}=0}^{k_1+k_2+\dots+k_{n-1}
%%\le\min\{r_1,r_2,\dots,r_n\}}
%\end{equation*}

\begin{rem}
For $d=2$, since, $D^l H_n(x) = 2^l\, l!\,
{n \choose l} H_{n-l}(x)$, therefore, we get
\begin{equation}\label{E:twocreations}
\begin{split}
(A_2^*)^{r_2}(A_1^*)^{r_1} e^{-x^tCx} 
=  (1+a_{11})^{\frac{r_1}{2}}(1+a_{22})^{\frac{r_2}{2}}
     \sum_{k_1=0}^{\min\{r_1,r_2\}}  {r_1 \choose k_1} {r_2 \choose k_1}
    \frac{(2a_{12})^k_1 k_1!}{\((1+a_{11})(1+a_{22})\)^{\frac{k_1}{2}}}\\
     \cdot H_{r_2-k_1}(S_2(x)) H_{r_1-k_1}\(S_1(x)\)e^{-x^tCx}.
\quad\qquad\qquad\qquad\qquad\qquad\quad
\end{split}
\end{equation}
\end{rem}

\subsection{Proof of formula (\ref{E:GCMF})}
We need the following Lemma. 
I could not find any reference of it, so we give a proof here.
\begin{lem}\label{L:BG}
Let $A$ be a real symmetric matrix of order $d$.
If $g_A(x) = e^{-x^tAx}$, then 
$Bg_A(z) = \(\frac{2}{\pi}\)^{d/4} 
                  \(\pi^d\det{(I+A)^{-1}}\)^{1/2} 
                  e^{\frac{1}{2}z^t\(\frac{I-A}{I+A}\)z}$.
\end{lem}
\begin{proof}
We compute
\begin{equation*}
\begin{aligned}
Bf(z) =&\; \(\frac{2}{\pi}\)^{d/4} e^{-\frac{1}{2}\norm{z}^2}
                 \int_{\R^n} e^{2x \cdot z - \abs{x}^2-x^tAx}\,dx\\
        =&\;  \(\frac{2}{\pi}\)^{d/4} e^{-\frac{1}{2}\norm{z}^2}
                 \int_{\R^n} e^{ -x^t(I+A)x + 2x \cdot z}\,dx\\
        =&\; \(\frac{2}{\pi}\)^{d/4} e^{-\frac{1}{2}\norm{z}^2}
                 \(\frac{\pi^d}{\det{(I+A)}}\)^{1/2}
                  e^{z^t(I+A)^{-1}z}\\
        =&\;  \(\frac{2}{\pi}\)^{d/4} \(\pi^d\det{(I+A)^{-1}}\)^{1/2}
                  e^{\frac{1}{2}z^t\(\frac{I-A}{I+A}\)z}.
\end{aligned}
\end{equation*}
\end{proof}
Now, we give the argument to justify the formula (\ref{E:GCMF}). Let 
\begin{equation*}
H(x) = \sum_{m,n,p=0}^{\infty} H_{n+p+r}(x_1) H_{p+m+s}(x_2) 
H_{m+n+t}(x_3)\frac{u_1^m}{m!}\frac{u_2^n}{n!}\frac{u_3^p}{p!},
\end{equation*}
and $h(x) = e^{-\abs{x}^2}H(\sqrt{2}x)$. 
Let $M = 
\begin{bmatrix}
    0       & u_3 & u_2 \\
    u_3   &   0   & u_1 \\
    u_2   & u_1 &   0
\end{bmatrix}.
$
We have
\begin{equation*}
h(x) = \sum_{m,n,p=0}^{\infty} K(m,n,p,r,s,t)
          2^{\frac{n+p+r}{2}}h_{n+p+r}(x_1) 
         2^{\frac{p+m+s}{2}}h_{p+m+s}(x_2) 
         2^{\frac{m+n+t}{2}}h_{m+n+t}(x_3)
         \frac{u_1^m}{m!}\frac{u_2^n}{n!}\frac{u_3^p}{p!},
\end{equation*}
where $K(m,n,p,r,s,t) = \(\frac{\pi}{2}\)^{3/4}
                                    \sqrt{(n+p+r)!(p+m+s)!(m+n+t)!}$.
Thus, from equation (\ref{E:BTdecomposition}), we get
\begin{equation*}
\begin{aligned}
Bh(z) = &\; 2^{\frac{r+s+t}{2}}\sum_{m,n,p=0}^{\infty} K
                  2^{n+m+p}B_1h_{n+p+r}(z_1) 
                  B_2h_{p+m+s}(z_2) B_3h_{m+n+t}(z_3)
                \frac{u_1^m}{m!}\frac{u_2^n}{n!}\frac{u_3^p}{p!}\\
        =&\; 2^{\frac{r+s+t}{2}} \(\frac{\pi}{2}\)^{3/4}
                \sum_{m,n,p=0}^{\infty}
                2^{n+m+p}z_1^{n+p+r} z_2^{p+m+s} z_3^{m+n+t} 
                \frac{u_1^m}{m!}\frac{u_2^n}{n!}\frac{u_3^p}{p!}\\
        =&\; 2^{\frac{r+s+t}{2}}  \(\frac{\pi}{2}\)^{3/4} z_1^r\,  
                z_2^s\,z_3^t\sum_{m,n,p=0}^{\infty} 
               \frac{(2z_1z_3u_2)^n}{n!} \frac{(2z_2z_3u_1)^m}{m!} 
               \frac{(2z_1z_2u_3)^p}{p!}\\
       =&\; 2^{\frac{r+s+t}{2}} \(\frac{\pi}{2}\)^{3/4} 
               z_1^r\,z_2^s\, z_3^t\, 
               \exp{\[2(u_3z_1z_2+u_1z_2z_3+u_2z_3z_1)\]}\\
       =&\; 2^{\frac{r+s+t}{2}} \(\frac{\pi}{2}\)^{3/4} 
               z_1^r\, z_2^s\, z_3^t\, e^{z^tMz}.
\end{aligned}
\end{equation*}
Therefore, from Lemma \ref{L:BG}, equation (\ref{E:inverse}), and 
(\ref{E:image}), we retrive
\begin{equation*}
h(x) = 2^{-\(\frac{r+s+t}{2}\)} \(\det{\(I+2M\)}\)^{-1/2}
          (A_1^*)^r (A_2^*)^s (A_3^*)^t
          e^{-x^t\(\frac{I-2M}{I+2M}\)x}.
\end{equation*}
Observe that 
$$\frac{I-2M}{I+2M} = -I + \frac{2I}{I+2M} = -I +
\frac{2}{\Delta}
\begin{bmatrix}
    1-4u_1^2       & 4u_1u_2-2u_3 & 4u_1u_3-2u_2 \\
    4u_1u_2-2u_3  &   1-4u_2^2   & 4u_2u_3-2u_1 \\
    4u_1u_3-2u_2   & 4u_2u_3-2u_1 &   1-u_3^2
\end{bmatrix}.$$
To apply, Theorem \ref{T:creation}, we note the entries of the matrix
$A = \frac{I-2M}{I+2M}$. For $i,j,k\in\{1,2,3\}$, we have
$$a_{ii} = -1 + \frac{2(1-4u_i^2)}{\Delta}
\quad\text{and}\quad
a_{ij} = -\frac{4(u_k-2u_iu_j)}{\Delta},\,\, i\neq j\neq k\neq i,
$$
and $S_{i}(x) = X_i(x)$.  
Thus, from Theorem \ref{T:creation} (for $d=3$), we get
\begin{equation*}
\begin{split}
H(x)
= \prod_{i=1}^3 \Delta^{-\(\frac{r_i+1}{2}\)}(1-4u_i^2)^{\frac{r_i}{2}}
\qquad\qquad\qquad\qquad\qquad \qquad 
\qquad \qquad \qquad
\qquad \qquad \qquad\quad
\qquad\qquad\qquad\qquad\quad\,\qquad\\ 
\cdot \exp\(\sum x_i^2 - \frac{\sum x_i^2 - 4\sum x_i^2 u_i^2 - 
4\sum_{i\neq j\neq k} x_i x_j u_k + 8\sum_{i\neq j} u_i u_j x_i x_j}{\Delta}\)
\qquad\qquad\quad\qquad \qquad \qquad
\qquad\qquad\,\,\\
\cdot \sum_{k\in\Lambda_3} \prod_{i=1}^{2} 
\(- \sqrt{\frac{\Delta}{2(1-4u_i^2)}}\)^{k_i} {r_{i+1} \choose k_i}
H_{r_3-k_2}(X_3(x))D_3^{k_2}\[
H_{r_2-k_1}(X_2(x))D_{2}^{k_1} H_{r_1}\(X_1(x)\)\].
\qquad\qquad\qquad\qquad\,\,\,\\
%\cdot D_3^{k_2} \[ H_{r_2-k_1}\(\frac{(y-2uz)(1-4v^2)-2(x-2vz)(w-2uv)}
%{(d(1-4v^2))^{1/2}}\)
%\qquad\qquad\qquad\qquad\qquad\qquad \qquad\\
%\cdot D_2^{k_1}H_{r_1}\(\frac{(x-2vz)(1-4u^2)-2(y-2uz)(w-2uv)}
%{(d(1-4u^2))^{1/2}}\) \]
%\qquad\qquad\qquad\qquad\qquad\qquad \qquad
\end{split}
\end{equation*}
%where
%\begin{equation*}
%\begin{split}
%P_{k_1,k_2}(x) = H_{r_3-k_2}D_3^{k_2} \[ H_{r_2-k_1}
%\(\frac{(y-2uz)(1-4v^2)-2(x-2vz)(w-2uv)}{(d(1-4v^2))^{1/2}}\)
%D_2^{k_1}H_{r_1}\(\frac{(x-2vz)(1-4u^2)-2(y-2uz)(w-2uv)}
%{(d(1-4u^2))^{1/2}}\) \]
%\end{split}
%\end{equation*}
This proves the formula (\ref{E:GCMF}). 
\qed

\begin{rem}
Similarly, We can obtain
the formula (\ref{E:srivastava}) from equation (\ref{E:twocreations}).
\end{rem}
%Note that $d=1$ case can be answered by equation (\ref{E:onecreation}).

\section*{Declarations}
I hereby declare that this work has no conflict of interest-neither
personal nor financial.

\bibliographystyle{amsplain}
\bibliography{v0-gmf}

\end{document}